\documentclass[11pt, letterpaper, reqno]{amsart}

\usepackage{amsmath,amsthm,amsrefs,amssymb,hyperref,caption,color}
\usepackage[percent]{overpic}
\usepackage{mathtools}
\usepackage{breqn}

\numberwithin{equation}{section}

\newtheorem{lemma}{Lemma}[section]
\newtheorem{theorem}{Theorem}[section]

\newtheorem{definition}{Definition}[section]

\theoremstyle{definition}
\newtheorem{remark}{Remark}[section]

\begin{document}
	\title[Rough isometry  and unbounded uniformizations]{Rough isometry between Gromov hyperbolic spaces and unbounded uniformization}
	
	\author[Vasudevarao Allu]{Vasudevarao Allu}
\address{Vasudevarao Allu, Department of Mathematics, School of Basic Sciences, 
Indian Institute of Technology Bhubaneswar,
Bhubaneswar-752050, 
Odisha, India.}
\email{avrao@iitbbs.ac.in}
\author[Alan P Jose]{Alan P Jose}
\address{Alan P Jose, Department of Mathematics, School of Basic Sciences, Indian Institute of Technology Bhubaneswar,
Bhubaneswar-752050, Odisha, India.}
\email{alanpjose@gmail.com}

	
	\date{\today}
	
	\begin{abstract} 
In a recent paper, Zhou, Ponnusamy, and Rasila [\textit{Math. Nachr.} (2025)] have established that the conformal deformations, with parameter $\epsilon>0$, of a Gromov hyperbolic space via Busemann functions are uniform spaces for sufficiently small $\epsilon$.
In this paper, we demonstrate that if two proper, roughly starlike Gromov hyperbolic spaces are roughly isometric, then the uniformity of their conformal deformations is a simultaneous property; that is, either both are uniform spaces or neither is. 
Our results provide a counterpart to the work of Lindquist and Shanmugalingam [\textit{Ann. Fenn. Math.} (2021)].
	\end{abstract}
	
	\maketitle
	
	\section{Introduction}\label{int}
A locally compact, rectifiably connected, non-complete metric space $X$ is said to be uniform if any two points in $X$ can be joined by a curve that does not go too near the boundary and whose length is comparable to the distance between the points. 
For a precise definition, we refer to Definition \ref{allu_jose_004_001ab}.
Uniform domains were introduced by Martio and Sarvas \cite{martio_sarvas_} in the setting of Euclidean spaces. 
The article by V\"ais\"al\"a \cite{vaisalaunif} provides an exhaustive treatment of various possible definitions of uniform domains. 
The importance of the class of uniform domains is well established, for instance P. W. Jones has established that uniform domains are Sobolev extension domains, and characterized them as BMO extension domains, (see \cites{jones_1980, jones_1981}). 
Further, the class of uniform domains has been studied extensively in geometric function theory, potential theory, and analysis on metric spaces (see \cites{geh87, vaisalaunif, BHK, HSX, BS07, koskela_lammi_2013} and the references therein).\vspace{2mm} 

Bonk, Heinonen, and Koskela \cite{BHK} in their seminal work demonstrated that there is a one-to-one correspondence between the quasi-isometry classes of proper geodesic roughly starlike Gromov hyperbolic spaces and the quasisimilarity classes of bounded locally compact uniform spaces.   
Moreover, they developed a uniformization theory which illustrated that the conformal deformations of a proper geodesic Gromov hyperbolic space are bounded uniform spaces. 
The theory of Gromov hyperbolic spaces originates in Gromov’s influential work \cite{gromov} on hyperbolic groups and has since become a central theme in analysis and geometry. 
In geometric function theory, they provide a natural setting for the study of quasiconformal mappings, offering a flexible framework to extend classical results beyond the Euclidean context  (see \cites{bonk_1996, bonk_schramm, HSX, zhou_li_rasila_2022, zhou_rasila_2022}). 
A geodesic metric space is Gromov hyperbolic if every geodesic triangle is thin, in the sense that any point on one edge is not too far from some point on the other two sides of the geodesic triangle.  
Note that Gromov hyperbolicity can be defined for spaces which are not necessarily geodesic (see Definition \ref{allu_jose_004_001c}), we refer to the book \cite{bridson} for an introduction into metric spaces of non-negative curvature. In this paper, we follow the general definition of Gromov hyperbolicity given in Definition \ref{allu_jose_004_001c}.
\vspace{2mm}

In 2001, Bonk, Heinonen, and Koskela \cite{BHK} used conformal densities which depends upon distance from a fixed point in the space. 
Recently, Zhou, Ponnusamy, and Rasila \cite{zhou_ponnusamy_rasila_2025} studied conformal densities via Busemann functions (see Subsection \ref{allu_jose_004_003aa} for definition) and proved that the conformal deformations $X_\epsilon = (X, d_\epsilon)$ induced by the densities \eqref{allu_jose_004_005} of a proper geodesic Gromov $\delta$-hyperbolic space $X$ that has atleast two points in the Gromov boundary are unbounded uniform spaces for all $0<\epsilon\leq \epsilon_0(\delta)$. 
It is worth to point out that the authors have developed uniformization techniques for hyperbolic spaces which are intrinsic but not necessarily geodesic, (see \cites{allu_jose_2, allu_jose_3}).
\vspace{2mm}

Let $(X, d_X)$ and $(Y, d_Y)$ be metric spaces. 
A map $\phi : Y \rightarrow X$ is a $\lambda$-rough isometry if 
\begin{equation}\label{allu_jose_004_001}
d_Y(y_1, y_2)-\lambda \leq d_X\left(\phi(y_1), \,\phi(y_2)\right) \leq d_Y(y_1, y_2)+\lambda
\end{equation} 
for all $y_1, y_2$ in $Y$ and $\phi(Y)$ is $\tau$-dense in $X$, that is for each $x$ in $X$ there exists a $y_x$ in $Y$ such that $d_X(x, \phi(y_x)) \leq \tau$. 
Moreover, it is known that (see \cite[Lemma 2.7]{lindquist_shanmu_2021}) there exists a $3\lambda$-rough isometry $\phi^{-1}: X \rightarrow Y$ such that for all $y\in Y$ and $x\in X$, we have 
\begin{equation}\label{allu_jose_004_001aa}
d_Y\left(y, \phi^{-1}\left(\phi(y)\right)\right) \leq 2\lambda,\,\,\,\,\, d_X\left(x, \phi\left(\phi^{-1}(x)\right)\right) \leq \lambda.
\end{equation}
Hence, by  replacing $3\lambda$ with $\lambda$ if needed, we will assume in the remainder of this paper that both $\phi$ and $\phi^{-1}$ are $\lambda$-rough isometries satisfying \eqref{allu_jose_004_001aa}.
\vspace{2mm}

If $X$ is a Gromov hyperbolic space , then the conformal deformations $X_\epsilon$ of $X$ via the Busemann functions are unbounded uniform spaces for sufficiently small $\epsilon$ by \cite[Theorem 1.2]{zhou_ponnusamy_rasila_2025}. 
The aim of this paper is to enquire whether the allowable range of uniformization parameters is preserved under rough isometries.
Our main theorem (Theorem \ref{allu_jose_004_024}) establishes that if  $\phi:Y\rightarrow X$ is a rough isometry, where $X$ is Gromov hyperbolic, and $\epsilon>0$ is such that $X_\epsilon$ is a uniform space, then $Y_\epsilon$ is also a uniform space. 
It should be noted that $Y_\epsilon$ is uniform for sufficiently small $\epsilon$ by the results in \cite{zhou_ponnusamy_rasila_2025}. 
The restriction on the parameter $\epsilon >0$ is due to a Gehring-Hayman inequality for conformal deformations \cite[Theorem 5.1]{BHK} that holds for sufficiently small $\epsilon$, which is a crucial tool in both \cites{BHK, zhou_ponnusamy_rasila_2025}. 
\vspace{2mm}

Note that the question of whether the allowable range of uniformizing parameters is preserved under rough isometries was addressed in \cite{lindquist_shanmu_2021}, where uniformization procedure was due to Bonk, Heinonen, and Koskela \cite{BHK}. \
In this paper, we examine the same question using the uniformization procedure of Zhou, Ponnusamy, and Rasila \cite{zhou_ponnusamy_rasila_2025}. 
We emphasize that the results of \cite{lindquist_shanmu_2021} do not apply in our setting, since here the conformal densities arise from Busemann functions.
\vspace{2mm}

The remainder of this paper is organized as follows. Section \ref{allu_jose_004_001abc}  is dedicated to provide relevant definitions and important preliminary results. In Section \ref{allu_jose_004_009aa}, we develop key lemmas necessary for the proof of the main result. The statement and proof of our main result, Theorem \ref{allu_jose_004_024}, are presented in Section \ref{allu_jose_004_024a}. Finally, in Section \ref{allu_jose_004_026}, a counterexample is provided to show that the conclusion of Theorem \ref{allu_jose_004_024} fails when rough isometries are replaced by rough similarities.

\section{Preliminaries}\label{allu_jose_004_001abc}
\subsection{Notations}
We follow the notation from \cite{buyalo}. 
For real numbers $a$ and $b$, we write $a \doteq b$ up to an additive error $\leq c$, or $a \doteq_c b$, in place of $|a-b| \leq c$. 
Similarly, we write $A \lesssim B$ if there exists a constant $C>0$ such that $A \leq C B$, and $A \gtrsim B$ if $B \lesssim A$. 
We write $A \asymp B$ up to a multiplicative error $\leq C$, or $A \asymp_C B$, in place of $C^{-1} B \leq A \leq C B$. 
When the exact constant is either irrelevant or clear from the context, we simply write $a \doteq b$ and $A \asymp B$.
Throughout this paper, for a metric space $(X,d_X)$, a point $x\in X$, and $r>0$, we denote
\[
B(x,r)=\{z\in X : d_X(x,z)<r\}, \qquad 
\overline{B}(x,r)=\{z\in X : d_X(x,z)\leq r\}.
\]
\begin{definition}\label{allu_jose_004_001ab}
Let $(X, d_X)$ be a locally compact, non-complete metric space, and set $\partial X := \overline{X}\setminus X$ be the metric boundary of $X$, where $\overline{X}$ denotes the metric completion of $X$. The space $(X, d_X) $ is called $A$-uniform if for each points $x, y\in X$ there exists a rectifiable curve $\gamma$ joining them satisfying 
\begin{equation}\label{allu_jose_004_001a}
 l(\gamma) \leq A d_X(x, y),
\end{equation}
\begin{equation}\label{allu_jose_004_001b}
\min\left\{l\left(\gamma[x, z]\right),\, l\left(\gamma[z, y]\right)\right\} \leq A \delta_X(z),
\end{equation}
where $\delta_X(z)= d_X(z, \partial X)$.  Here $\gamma[x, z]$ denotes the subcurve of $\gamma$ with endpoints $x, z$.
\end{definition}
A curve satisfying \eqref{allu_jose_004_001a} is called a quasiconvex curve. 
Condition \eqref{allu_jose_004_001b} is referred to  as the double cone condition.
Note that there are definitions of uniform domains that do not require local compactness; see \cite{vaisalaunif} for a detailed discussion of the various formulations.

\subsection{Gromov hyperbolic spaces}
In this paper we only consider proper (closed balls are compact) geodesic metric spaces. 
\begin{definition}\label{allu_jose_004_001c}
Let $(X, d_X)$ be a metric space and $o\in X$ be fixed. For $x, y \in X$, put 
\[
(x|y)_o = \frac{1}{2}\left( d_X(x, o)+ d_X(y, o) - d_X(x, y)\right).
\]
The number $(x|y)_o$ is non-negative by the triangle inequality, and it is called the Gromov product of $x, y$ with respect to $o$. Let $\delta \geq 0$.  We say that $X$ is Gromov $\delta$-hyperbolic or  $\delta$-hyperbolic if 
\[
(x|y)_o \geq \min \left\{(x|z)_o, (z|y)_o\right\} -\delta
\]
for all $x, y, z, o \in X$.
\end{definition}
The above definition of Gromov hyperbolicity holds in arbitrary metric spaces. If the space is geodesic or intrinsic, this is equivalently characterized by the Rips condition (see \cite[Propositions 2.1.2 and 2.1.3]{buyalo} and \cite[Theorems 2.34 and 2.35]{vaisala_2004_gh}).
\vspace{2mm}

Let $X$ be a hyperbolic space and $o\in X$ be fixed. A sequence $\left\{x_n\right\}$ in $X$ is said to be a Gromov sequence if 
\[
\lim_{m, n \rightarrow \infty} (x_m|x_n)_o = \infty.
\]
Two Gromov sequences $\{x_n\}, \{x^{\prime}_n\}$ are equivalent if 
\[
\lim_{m, n \rightarrow \infty} (x_n|x^{\prime}_n)_o = \infty.
\]
The \textit{Gromov boundary} $\partial_\infty X$ of $X$ is defined as the set of all equivalence classes of Gromov sequences. 
\vspace{2mm}

The Gromov product can be extended to the points in the Gromov boundary. For points $\xi, \xi^{\prime} \in \partial_\infty X, x\in X$ and for a fixed base point $o\in X$, the Gromov product is defined as 
\[
(x|\xi)_o  = \inf \left\{\liminf_{n\rightarrow \infty } (x|x_n): \{x_n\} \in \xi \right\}.
\]
Similarly,
\[
(\xi|\xi^{\prime})_o= \inf \left\{\liminf_{n\rightarrow \infty } (x_n|x^{\prime}_n)_o: \{x_n\} \in \xi, \{x^{\prime}_n\} \in \xi^{\prime} \right\}.
\]
We need the following lemma from \cite{vaisala_2004_gh} to prove our main result.
\begin{lemma} \cite[Lemma 5.11]{vaisala_2004_gh}\label{allu_jose_004_002}
Let $X$ be a Gromov $\delta$-hyperbolic space.
Suppose $\{x_n\} \in \xi \in \partial_\infty X, x^{\prime}_n \in \xi^{\prime} \in \partial_{\infty} X$ and $x, o\in X$, then
\begin{align*} 
&(\xi|\xi^{\prime})_o \leq \liminf_{n\rightarrow \infty } (x_n|x^{\prime}_n)_o \leq \limsup_{n\rightarrow \infty } (x_n|x^{\prime}_n)_o \leq (\xi|\xi^{\prime})_o +2\delta,\\
&(\xi|x)_o  \leq \liminf_{n\rightarrow \infty } (x_n|x)_o \leq \limsup_{n\rightarrow \infty } (x_n|x)_o \leq (\xi|x)_o +\delta.
\end{align*}
\end{lemma}
We state the definition of rough starlikeness as in \cite{zhou_ponnusamy_rasila_2025}  (see also \cite{vaisala_2004_gh}).
\begin{definition}
A proper geodesic $\delta$-hyperbolic space $(X, d_X)$ is said to be $K$-\textit{roughly starlike}, $K\geq0$, with respect to a base point $\xi \in \partial_\infty X$, if for every point $x\in X$ there exists $\zeta \in \partial_\infty X$ and a geodesic line $\gamma$ joining $\xi$ and $\zeta$ such that  $\mathrm{dist}(x, \gamma) \leq K$.
\end{definition}
Metric trees are Gromov $\delta$-hyperbolic with $\delta =0$ and $K$-roughly starlike with $K=0$. 
The concept of rough starlikeness with respect to a base point in the space was introduced by Bonk, Heinonen, and Koskela \cite[Theorem 3.6]{BHK}. 
They showed that uniform spaces are Gromov hyperbolic when equipped with the quasihyperbolic metric and are roughly starlike with respect to points in the space, provided the space is bounded. 
Recently, Zhou, Ponnusamy, and Rasila  \cite[Lemma 3.14]{zhou_ponnusamy_rasila_2025} established that unbounded uniform spaces are roughly starlike with respect to each point in the Gromov boundary.

\subsection{Busemann functions}\label{allu_jose_004_003aa}
Let $(X, d_X)$ be a $\delta$-hyperbolic space. 
For a fixed base point $o\in X$ and $\xi \in \partial_\infty X$, let $\mathcal{B}(\xi)$ be the set of Busemann functions based at  $\xi \in \partial_\infty X$ as in \cite[Section 3.1]{buyalo}. 
The canonical function $b_{\xi, o}\in \mathcal{B}(\xi) $ is given by
\[
b(x)= b_{\xi, o}(x) = b_\xi(x, o) = (\xi | o)_x - (\xi | x)_o
\]
for every $x\in X$. \vspace{2mm}

As per \cite[Lemma 3.1.1]{buyalo}, we observe that
\begin{equation}\label{allu_jose_004_003a}
b_\xi (x, o) \doteq_{2\delta} (z_i|o)_x - (z_i|x)_o = \left\{ d_X(x, z_i)- d_X(o, z_i)\right\}_i
\end{equation}
for every sequence $(z_i)\in \xi$ and every $x, o\in X$.\vspace{2mm}

Busemann functions are roughly 1-Lipschitz (see \cite[Proposition 3.1.5(1)]{buyalo}), that is, for every $b\in \mathcal{B}(\xi)$ and for every $x, x^\prime\in X$ 
\begin{equation}\label{allu_jose_004_003}
|b(x)-b(x^\prime)|\leq d_X(x, x^\prime)+10\delta.
\end{equation}
Further, by \cite[Lemma 3.1.2]{buyalo}, we have 
\begin{equation}\label{allu_jose_004_004}
b_\xi(x, o) -b_\xi (x, o^\prime) \doteq_{6\delta} b_\xi(o, o^\prime)
\end{equation}
for every $x, o, o^\prime \in X$. 

\subsection{Conformal deformations}

Let $(X, d_X)$ be a proper $\delta$-hyperbolic space. 
Fix a Busemann function $b$ based on $\xi \in \partial_\infty X$, $b=b_{\xi, o}\in \mathcal{B}(\xi)$ with $o\in X$.
For $\epsilon>0$, consider the family of conformal deformations of $X$ generated by the densities 
\begin{equation}\label{allu_jose_004_005}
\rho_\epsilon(x)= \exp \left\{-\epsilon b(z)\right\}.
\end{equation} 
For $x, x^\prime\in X$, let
\[
d_\epsilon(x, x^\prime) = \inf_\gamma \int_\gamma \rho_\epsilon \mathrm{ds},
\]
where the infimum is taken over all rectifiable curves $\gamma$ in $(X, d_X)$ with $x$ and $y$ as the endpoints. 
It is easy to see that $d_\epsilon$ is a metric on $X$.
We denote the resulting metric space by $X_\epsilon= (X, d_\epsilon)$.
For a rectifiable curve $\gamma$ in $X$, the length of $\gamma$ in the space $(X, d_\epsilon)$ is denoted by $l_\epsilon(\gamma)$.
It is known that (see \cite[Lemma 2.6]{BHK}) if the identity map $\left(X, d_X\right) \to \left(X,  l_{d_X}\right)$ is a homeomorphism then for any rectifiable curve $\gamma$ we have
\[
l_\epsilon (\gamma) = \int_\gamma \rho_\epsilon^X ds.
\]
Here $l_{d_X}$ is the inner metric of $X$. 
For $x, x^\prime\in X$, by virtue of \eqref{allu_jose_004_003}, we have the following \textit{Harnack type inequality}:
 
\begin{equation}\label{allu_jose_004_006}
\exp\left\{ -10\epsilon\delta- \epsilon d_X(x, x^\prime)\right\} \leq \frac{\rho_\epsilon(x)}{\rho_\epsilon(x^\prime)} \leq \exp\left\{ 10\epsilon\delta+ \epsilon d_X(x, x^\prime)\right\}.
\end{equation}
From \eqref{allu_jose_004_006}, we deduce that whenever $d_X(x, x^\prime) \leq 1$, we have 
\begin{equation}\label{allu_jose_004_006aa}
\frac{1}{C}\leq\frac{\rho_\epsilon(x)}{\rho_\epsilon(x^\prime)}\leq C,
\end{equation}
with $C= \exp\left\{10\epsilon\delta+\epsilon\right\}$.
\vspace{2mm}

We now state the following lemma from \cite{lindquist_shanmu_2021}.
\begin{lemma}\cite[Lemma 3.1]{lindquist_shanmu_2021}\label{allu_jose_004_006a}
Suppose that $\rho : Y \rightarrow (0, \infty)$ satisfies the Harnack type inequality \eqref{allu_jose_004_006aa} with constant $C$. Let $L>1$ and $\gamma :[0, L] \rightarrow Y$ be a curve with $l(\gamma)=L$. Choose $N\in \mathbb{N}$ such that $N\leq L < N+1$. Then
\begin{equation}\label{allu_jose_004_006b}
\int_\gamma \rho \,\mathrm{ds} \asymp \sum_{i=0}^{N-1} \rho(a_i),
\end{equation}
where $a_i=\gamma(iq)$ with $q= \frac{L}{N}$. The comparison constant in \eqref{allu_jose_004_006b} can be taken to be $2C^2$. If $L\leq Q$ with $Q\geq 1$ a fixed number, we instead have 
\[ 
\int_\gamma \rho \,\mathrm{ds} \asymp L \rho\left(\gamma(0)\right)
\]
with comparison constant $C^{Q+1}$.
\end{lemma}
\begin{remark}\label{allu_jose_004_007a}
It is straightforward to verify that if $\gamma \colon [0,\infty) \to Y$ is a curve, then the conclusion of Lemma \ref{allu_jose_004_006a} remains valid with $a_i = \gamma(iq)$, where $q\geq 1$, and the sum is taken over all $i \geq 0$.
\end{remark}
\begin{remark}\label{allu_jose_004_007}
Let $X$ be a $\delta$-hyperbolic space and $\phi: Y \rightarrow X$ be a $\lambda$-rough isometry. 
Then it is easy to see that $Y$ is also a $\delta^\prime$-hyperbolic space with constant $\delta^\prime = 3\lambda+\delta$. 
Moreover, if $(X, d_X)$ is a proper Gromov $\delta$-hyperbolic space and $\partial_\infty X$ contains at least two points, then by \cite[Theorem 1.2]{zhou_ponnusamy_rasila_2025} the conformal deformations $X_\epsilon$ induced by the densities \eqref{allu_jose_004_005} are unbounded locally compact uniform spaces for $0<\epsilon \leq \epsilon_0$.
In this paper we do not put this restriction on the parameter, rather we use the technique of discretization of paths. 
\end{remark}
It is worth to note that the authors have developed a Gehring-Hayman type theorem for conformal deformations of Gromov hyperbolic space which are not necessarily geodesic (see \cite[Theorem 1.3]{allu_jose_2}). 
\begin{remark}\label{allu_jose_004_008}
Let $X$  be a hyperbolic space and $\phi : X\rightarrow Y$ be a $\lambda$-rough isometry with $\lambda$-rough isometric inverse. 
Observe that a sequence $(x_n)$ in $X$ is a Gromov sequence if, and only if, $\left(\phi(x_n)\right)$ is a Gromov sequence in $Y$. 
Therefore, there is a one-to-one correspondence between $\partial_\infty X$ and $\partial_\infty Y$ by the map $\phi$. 
Hence, for $\omega\in \partial_\infty X$ we denote the corresponding element in $\partial_\infty Y$ by $\omega^{\prime}$. 
Also, for the Busemann function $b=b_{\omega, o} \in \mathcal{B}(\omega)$, we denote  the corresponding Busemann function as $b^\prime = b_{\omega^\prime\!, o'}\in \mathcal{B}(\omega^\prime)$, where $o^\prime=\phi(o)$.
Building on these notations, we are now concerned about the densities
\[
\rho_\epsilon^X(x) = \exp\left\{-\epsilon b(x)\right\} \text{ and } 
\rho_\epsilon^Y(y) = \exp\left\{-\epsilon b^\prime(y)\right\}
\]
for $x\in X, y\in Y$ and $\epsilon>0$.
\end{remark}

 \section{Key lemmas}\label{allu_jose_004_009aa}
%
%

Let $(X, d_X)$ and $(Y, d_Y)$ be proper metric spaces, and let $\phi:Y \rightarrow X$ be a $\lambda$-rough isometry. As discussed in the introduction, we assume that $\phi^{-1}$ is also a $\lambda$-rough isometry.
\vspace{2mm}

If $(X, d_X)$ is $\delta$-hyperbolic for some $\delta \geq 0$, then by Remark \ref{allu_jose_004_007}, $(Y, d_Y)$ is $(3\lambda+\delta)$-hyperbolic. To simplify our discussion, we may assume for the rest of this paper that both $X$ and $Y$ are proper $\delta$-hyperbolic spaces. \vspace{2mm}

For a point $\omega \in \partial_\infty Y$ and a base point $o \in Y$, let $\omega' = \phi(\omega)$ and $o' = \phi(o)$. 
For any Busemann  function $b = b_{\omega, o} \in \mathcal{B}(\omega)$, we define its image as $b' = \phi(b) = b_{\omega'\!, o'} \in \mathcal{B}(\omega')$ (see Remark \ref{allu_jose_004_008}). 
Unless otherwise indicated, all curves in the following are assumed to be parametrised by arclength.
\vspace{2mm}

First, we establish a lemma that will be useful in the proofs of the forthcoming results.
\begin{lemma}\label{allu_jose_004_009a}
Let $Y$ be a Gromov $\delta$-hyperbolic space and $b\coloneq b_{\omega, o}$ be a Busemann function based on $\omega\in \partial_\infty Y$ and $o\in Y$. Furthermore, $b^\prime = \phi(b)= b_{\omega^\prime\!, o^\prime} \in \mathcal{B}(\omega^\prime)$ be the image of $b$ under $\phi$. 
Then for any $x, z\in Y$ the following holds:
\begin{enumerate}
\item[(a)] \label{allu_jose_004_009b} $b(x) - d_Y\!(x, z)  \leq b(z) \leq b(x) +d_Y\!(x, z),$
\item[(b)] \label{allu_jose_004_009c} $b(x) \doteq_{5\lambda} b^\prime(\phi(x))$.
\end{enumerate}
\end{lemma}
\begin{proof}
For any sequence $\left(w_n\right) \in \omega$ we have
\begin{align*}
\left(w_n |o\right)_z  &= \frac{1}{2} \left( d_Y\!\left(w_n, z\right) + d_Y\!\left(o, z\right) - d_Y\!\left(w_n, o\right)\right)\\
&\leq \frac{1}{2} \left( d_Y\!\left(w_n, x\right) + d_Y\!\left(x, z\right)+ d_Y\!\left(o, x\right)+ d_Y\!\left(x, z\right) - d_Y\!\left(w_n, o\right)\right)\\
& = \left(w_n |o\right)_x +  d_Y\!\left(x, z\right).
\end{align*}
Similarly, we obtain
\[
\left(w_n |o\right)_z   \geq \left(w_n |o\right)_x -  d_Y\!\left(x, z\right)\!.
\]
Thus we conclude that 
\[ 
\left(w_n |o\right)_x -  d_Y\!\left(x, z\right) \leq  \left(w_n |o\right)_z   \leq \left(w_n |o\right)_x +  d_Y\!\left(x, z\right)\!.
\]
Proceeding analogously, we obtain the following relation:
\[
\left(w_n |x\right)_o - d_Y\!\left(x, z\right)\leq 
\left(w_n |z\right)_o \leq \left(w_n |x\right)_o +d_Y\!\left(x, z\right).
\]
Therefore, we obtain 
\[
\left(\omega |o\right)_x -  d_Y\!\left(x, z\right) \leq  \left(\omega |o\right)_z   \leq \left(\omega |o\right)_x +  d_Y\!\left(x, z\right),
\]
and 
\[
\left(\omega |x\right)_o - d_Y\!\left(x, z\right)\leq 
\left(\omega |z\right)_o \leq \left(\omega |x\right)_o +d_Y\!\left(x, z\right).
\]
Consequently, by the definition of Busemann functions, we have 
\[
b(x) -2\, d_Y\!(x, z) \leq b(z)\leq b(x) +2\, d_Y\!(x, z).
\]
This completes the proof of part (a).\vspace{2mm}

\vspace{2mm}
Recall that 
\[ 
(\omega^\prime|o^\prime)_{x^\prime} = \inf \left\{\liminf_{n \to \infty} (z_n^\prime|o^\prime)_{x^\prime} \colon (z_n^\prime)\in \omega^\prime\right\}\!.
\]
Our first aim is to show that the sequences $(z_n^\prime)\in \omega^\prime$ can equivalently be taken in the form $\phi(z_n)$, where $(z_n)\in \omega$, and that the corresponding Gromov product is comparable to $(\omega^\prime|o^\prime)_{x^\prime}$.
Indeed, let $(z_n)$ be a Gromov sequence in $\omega$. Then $(\phi(z_n))\in \omega^\prime$. Conversely, since $\phi^{-1}$ is a $\lambda$-rough isometry, if $z_n^\prime \in \omega^\prime$ then $z_n = \phi^{-1}(z_n^\prime)\in \omega$.
Note that, although $\phi(z_n)$ may not coincide with $z_n^\prime$, they satisfy
\[
d_X(z_n^\prime, \phi(z_n)) \leq \lambda.
\]
We now claim that
\[
(z_n^\prime|o^\prime)_{x^\prime} \doteq_\lambda (\phi(z_n)|o^\prime)_{x^\prime}.
\]
By the triangle inequality and using $d_X(z_n^\prime, \phi(z_n)) \leq \lambda$, we compute
\begin{align*}
(z_n^\prime|o^\prime)_{x^\prime} 
&= \frac{1}{2} \left( d_X(z_n^\prime, x^\prime) + d_X(o^\prime, x^\prime) - d_X(z_n^\prime, o^\prime)\right) \\
&\doteq_\lambda \tfrac{1}{2} \left( d_X(\phi(z_n), x^\prime) + d_X(o^\prime, x^\prime) - d_X(\phi(z_n), o^\prime)\right) \\
&\doteq_\lambda (\phi(z_n)|o^\prime)_{x^\prime}.
\end{align*}
Therefore, we see that 
\begin{equation}\label{allu_jose_004_009d}
(\omega^\prime|o^\prime)_{x^\prime} \doteq_\lambda \inf \left\{\liminf_{n \to \infty} (\phi(w_n)|o^\prime)_{x^\prime} \colon (w_n)\in \omega\right\}\!.
\end{equation}
Now, applying the $\lambda$-rough isometry property of $\phi$, we obtain
\[
(\phi(w_n)|o^\prime)_{x^\prime} \doteq_{\frac{3\lambda}{2}} (w_n|o)_x.
\]
Substituting this into \eqref{allu_jose_004_009d} gives
\[
(\omega^\prime|o^\prime)_{x^\prime} \doteq_{\frac{5\lambda}{2}} (\omega|o)_x.
\]
Finally, recalling the definition of the Busemann function, we conclude
\[
b(x) = b_{\omega,o}(x) \doteq_{5\lambda} b_{\omega^\prime,o^\prime}(\phi(x)) = b^\prime(\phi(x)),
\]
which is the desired result.
\end{proof}

\begin{lemma}\label{allu_jose_004_009e}
Let $x, y\in Y$ with $d_Y(x, y) >1$, and $\gamma:[0, L] \rightarrow Y$ be a curve from $x$ to $y$. 
Let $N\in \mathbb{N}$ be such that $N\leq L <N+1$. Then the following holds:
\begin{align*}
\int_\gamma \rho_\epsilon^Y \mathrm{ds} &\asymp \sum_{i=0}^{N-1} \rho_\epsilon^X \left(\phi(a_i)\right)\\
& \asymp \sum_{i=0}^{N-2} \rho_\epsilon^X \left(\phi(a_i)\right) +  \rho_\epsilon^X \left(\phi(y)\right)
\end{align*}
where $a_i= \gamma(iq) $ for $0\leq i \leq N$, and  $q= \frac{L}{N}$. 
If $N=1$, we follow the convention that the sum is empty and thus equal to zero.
\end{lemma}
\begin{proof}
From Lemma \ref{allu_jose_004_006a}, we have
\[
\int_\gamma \rho_\epsilon^Y \mathrm{ds} \asymp \sum_{i=0}^{N-1} \rho_\epsilon^Y(a_i)
\]
with constant $2C^2$, where $C$ is the constant from \eqref{allu_jose_004_006aa}.
Furthermore, for $(z_i)\in \omega$, $o, y\in Y$, it is easy to see that
\[ 
\left(\phi(z_i)|\phi(o)\right)_{\phi(y)} \doteq_{\frac{3}{2}\lambda} \left(z_i|o\right)_y.
\]
Similarly,
\[
\left(\phi(z_i)|\phi(y)\right)_{\phi(o)} \doteq_{\frac{3}{2}\lambda} \left(z_i|y\right)_o.
\]
Therefore, \eqref{allu_jose_004_003a} gives us
\begin{equation}\label{allu_jose_004_009}
b'\left(\phi(y)\right) \doteq_{3\lambda +4\delta} b(y).
\end{equation}
For $0\leq i\leq N$, we set $b_i = \phi(a_i)$. Then, \eqref{allu_jose_004_009} together with the notations from Remark \ref{allu_jose_004_008}, we obtain that
\[
\exp \left\{-\epsilon\left(4\delta+3\lambda\right)\right\} \leq \frac{\rho_\epsilon^Y(a_i)}{\rho_\epsilon^X(b_i)} \leq \exp \left\{\epsilon\left(4\delta+3\lambda\right)\right\}.
\]
Hence,
\[ 
\sum_{i=0}^{N-1} \rho_\epsilon^Y(a_i) \asymp \sum_{i=0}^{N-1} \rho_\epsilon^X(\phi(a_i))
\]
with constant $\exp\left\{\epsilon\left(4\delta+3\lambda\right)\right\}$.\\
Thus, 
\begin{equation}\label{allu_jose_004_010a}
\int_\gamma \rho_\epsilon^Y \mathrm{ds} \asymp \sum_{i=0}^{N-1} \rho_\epsilon^X(\phi(a_i))
\end{equation}
with constant $2C^2\exp\left\{ \epsilon \left(4\delta+3\lambda\right)\right\}$.
Now, note that $d_Y(a_{N-1}, y) \leq 1$, therefore from \eqref{allu_jose_004_006aa} we obtain that $\rho_\epsilon^Y \left(a_{N-1}\right) \asymp \rho_\epsilon^Y \left(y\right)$ with constant $C$. 
But $\rho_\epsilon^Y \left(y\right)\asymp \rho_\epsilon^X \left(\phi(y)\right)$ with constant $\exp\left\{\epsilon\left(4\delta+3\lambda\right)\right\}$.
 Combining this with \eqref{allu_jose_004_010a} we get the second comparability. This completes the proof.
\end{proof}

\begin{lemma}\label{allu_jose_004_010}
Let $Y$ be a Gromov $\delta$-hyperbolic space and $\epsilon>0$. Then for each $x\in Y$ we have 
\begin{align*}
\delta_\epsilon(x) &\coloneq \mathrm{dist} \left(x, \partial_\epsilon Y\right)\\
&\gtrsim \rho_\epsilon^Y(x). 
\end{align*}
Additionally, if $\partial_\infty Y$ has atleast two points, and $Y$ is $K$-roughly starlike with respect to $\zeta\in \partial_\infty Y\setminus \{\omega\}$, then 
\[
\delta_\epsilon(x) \lesssim \rho_\epsilon^Y(x).
\]
\end{lemma}

\begin{proof}
Let $x \in Y$ be fixed, and set $r = \tfrac{1}{\epsilon}$. For every point $y \in Y$ with $d_Y(x,y) \geq r$, and for every rectifiable curve $\gamma \colon [0,L] \to Y$ joining $x = \gamma(0)$ to $y = \gamma(L)$, the restriction $\gamma|_{[0,r]}$ is a subcurve of length $r$ starting at $x$.\\
Then,

\begin{align*}
d_\epsilon(x, y) &\geq \int_{0}^r \rho_\epsilon^Y\left(\gamma(s)\right) \mathrm{ds}\\
&\geq \int_0^{1/\epsilon} \exp\left\{ -10\epsilon\delta\right\} \rho_\epsilon^Y(x) e^{-\epsilon s} \mathrm{ds}\\
&= \exp\left\{-10\epsilon\delta\right\}\frac{\left(1-e^{-1}\right)}{\epsilon} \rho_\epsilon^Y(x).
\end{align*}
Thus, we have
\[
\delta_\epsilon(x) \geq \frac{\exp\left\{ -10\epsilon\delta\right\}}{2\epsilon} \rho_\epsilon^Y (x).
\]
\vspace{2mm}
Further, suppose that $Y$ is roughly $K$-starlike with respect to $\xi \in \partial_\infty Y$, where $\xi \neq \omega$. 
Then there exists a geodesic line $\eta \colon \mathbb{R} \to Y$ such that 
\[
\eta(-\infty) = \zeta, \quad \eta(\infty) = \xi, \quad \text{and} \quad d_Y\! \left(x, \eta(t_0)\right) \leq K.
\]
Now, from \eqref{allu_jose_004_004} and \eqref{allu_jose_004_003a} for any $\left(z_n\right) \in \omega$ we see that
\begin{align*}
b(z)-b(x) &\geq b_\omega \left(z, \eta\left(s_0\right)\right)-  b_\omega \left(x,\eta\left(s_0\right)\right) -12\delta\\
&\geq \left\{ d_Y\!\left(z, z_i\right) - d_Y\!\left(\eta(s_0), z_i\right)\right\} - 
\left\{ d_Y\!\left(x, z_i\right) - d_Y\!\left(\eta(s_0), z_i\right)\right\} -16\delta\\
&\geq d_Y\!\left(\eta(s_0), z\right) - d_Y\!\left(x, \eta(s_0)\right) -16\delta\\
&\geq d_Y\!\left(\eta(s_0), z\right) - K -16\delta.
\end{align*}
It follows that 
\begin{equation}\label{allu_jose_004_011}
\rho_\epsilon^Y (z) \leq \rho_\epsilon^Y(x) \exp\left\{\epsilon\left(K + 16\delta\right)\right\} \exp\left\{ -\epsilon d_Y\!\left(\eta\left(s_0\right), z\right)\right\}\!.
\end{equation}
Next, let $\gamma$ be a hyperbolic geodesic connecting $x$ and $\eta (t_0)$. 
Set $\gamma^*$ as the concatenation of $\gamma$ and $\eta|_{[s_0, \infty]}$.
\vspace{2mm}

Then we have
\begin{align*}
\delta_\epsilon(x) &\leq \int_{\gamma^*} \rho_\epsilon^Y(z) \mathrm{|dz|}\\
&= \int_\gamma \rho_\epsilon^Y(z) \mathrm{|dz| } + \int_{\eta|_{[s_{0}, \infty]}} \rho_\epsilon^Y \left(z\right)\mathrm{|dz|}. 
\end{align*}
Note that for any $z\in\gamma $, we have $d_Y (x, z) \leq d_Y (x, \eta(s_0)) \leq K$. Combining this with Lemma \ref{allu_jose_004_009a}, we see that
\begin{equation}\label{allu_jose_004_012}
\int_\gamma \rho_\epsilon^Y(z) \mathrm{|dz| }  \leq \int_\gamma \rho_\epsilon^Y (x) e^{\epsilon K} \mathrm{|dz| } = K\rho_\epsilon^Y (x) e^{2K\epsilon  }.
\end{equation}
Subsequently, from \eqref{allu_jose_004_011} it follows that

\begin{align}
 \int_{\eta|_{[s_{0}, \infty]}} \rho_\epsilon^Y \left(z\right)\mathrm{|dz|} &\leq \rho_\epsilon^Y(x) \exp\left\{\epsilon\left(K + 16\delta\right)\right\} \int_0^\infty e^{-\epsilon t} \mathrm{dt} \nonumber\\
 &=  \epsilon^{-1}\rho_\epsilon^Y(x) \exp\left\{\epsilon\left(K + 16\delta\right)\right\}\label{allu_jose_004_013}.
\end{align}
Therefore, we conclude that 
$
\delta_\epsilon(x) \lesssim \rho_\epsilon^Y(x).
$\end{proof}
\begin{lemma}\label{allu_jose_004_014}
Suppose $X$ and $Y$ are Gromov $\delta$-hyperbolic spaces, and let $\phi\colon Y \to X$ be a $\lambda$-rough isometry. 
Then, for each $\epsilon > 0$ and every $y \in Y$, 
\[
\delta_\epsilon(y) \asymp \delta_\epsilon\left(\phi(y)\right).
\]
\end{lemma}

\begin{proof}
Let $y\in Y$ and $x=\phi(y)$. Let $\gamma \colon [0, \infty) \rightarrow Y$ be any path emanating from $y$ that leaves every compact subset of $Y$.
Define $a_0=y$, and for each $k\in \mathbb{N}$, set $a_k \coloneq \gamma\left(\left(1+\lambda\right)k\right)$.
Then by Lemma \ref{allu_jose_004_006a}, Remark \ref{allu_jose_004_007a} and from the fact that $\phi$ is rough isometry, we have 
\[ 
l_\epsilon(\gamma) \asymp \sum_{k=0}^\infty \rho_\epsilon^Y(a_k) \asymp \sum_{k=0}^\infty \rho_\epsilon^X\left(\phi\left(a_k\right)\right).
\]
Let $\beta_k$ be a geodesic joining $\phi(a_k) $ and $\phi(a_{k+1})$ and $\beta $ be the concatenation of $\beta_k$ for $k=0, 1, 2, \cdots$. 
It is easy to see that $d_Y\!\left(a_k, a_{k+1}\right) \leq 1+\lambda$ and thus $d_X\left(\phi\left(a_k\right), \phi\left(a_{k+1}\right)\right) \leq 1+2\lambda$.
Hence, from the second part of Lemma \ref{allu_jose_004_006a}, we deduce that
\[ 
l_\epsilon(\beta) =\sum_{k=0}^\infty l_\epsilon\left(\beta_k\right) \asymp \sum_{k=0}^\infty  \rho_\epsilon^X\left(\phi\left(a_k\right)\right)\asymp l_\epsilon(\gamma).
\]

Note that since $\gamma$ leaves every compact subset of $Y$ and $\phi$ is a rough isometry, it follows that $\beta$ leaves every compact subset of $X$. 
Consequently, using the fact that $X_\epsilon$ is an intrinsic (length) space, we obtain 
\[
\delta_\epsilon\left(\phi(y)\right) \leq l(\beta),
\]
for any such curve $\beta$.  
Since $Y_\epsilon$ is also an intrinsic space, taking the infimum over all curves $\gamma \colon [0,\infty) \to Y$ that leave every compact subset of $Y$ with $\gamma(0)=y$, and combining with the above inequality, we deduce that
\[
\delta_\epsilon\left(\phi(y)\right) \lesssim \delta_\epsilon(y).
\]
To obtain the reverse inequality, we interchange the roles of $X$ and $Y$ in the above argument and use $\phi^{-1}$ in place of $\phi$.  This  yields
\[
\delta_\epsilon\left(\phi^{-1}\left(\phi(y)\right)\right) \lesssim \delta_\epsilon(\phi(y)).
\]
Observe that from the definition of $\phi^{-1}$, we always have 
\[
d_Y\!\left(y, \phi^{-1}\left(\phi(y)\right)\right) \leq \lambda.
\]
Thus, we have
\begin{align}
\delta_\epsilon(y) &\leq \delta_\epsilon\left(\phi^{-1}\left(\phi(y)\right)\right) + d_\epsilon \left(y, \phi^{-1}\left(\phi(y)\right)\right)\nonumber\\
& \lesssim \delta_\epsilon\left(\phi^{-1}\left(\phi(y)\right)\right) + \tau \rho_\epsilon^Y\left(\phi^{-1}\left(\phi(y)\right)\right) \nonumber\\
&\lesssim \delta_\epsilon\left(\phi^{-1}\left(\phi(y)\right)\right)\label{allu_jose_004_015a},
\end{align}
where we have used Lemma \ref{allu_jose_004_010} and second part of Lemma \ref{allu_jose_004_006a}. This completes the proof.
\end{proof}
\begin{lemma}\label{allu_jose_004_015}
Let $x, y \in Y$ satisfy $d_Y(x, y) \leq 4+\lambda$, and let $\gamma$ be a hyperbolic geodesic in $Y$ joining $x$ and $y$. Then
\[
l_\epsilon(\gamma) \asymp d_\epsilon(x, y) \asymp \rho_\epsilon^Y(x)d_Y(x, y).
\]
Moreover, $\gamma$ is a uniform curve in $Y_\epsilon$, with uniformity constant depending on $\epsilon$, $\lambda$, and $\delta$.
\end{lemma}
\begin{proof}
The length of $\gamma$ in the uniformized space $Y_\epsilon$ is given by 
\[
l_\epsilon(\gamma) = \int_\gamma \exp\{\epsilon b(z)\}\, |dz|.
\]
Furthermore, since $d_Y(x, y) \leq 4+\lambda$, by Lemma~\ref{allu_jose_004_009a} we have, for any $z \in \gamma$,
\begin{equation}\label{allu_jose_004_016}
b(z) \doteq_{2\lambda + 8} b(x).
\end{equation}
Consequently, we obtain 
\begin{equation}\label{allu_jose_004_017}
\exp\{-2\lambda\epsilon-8\epsilon\} e^{-\epsilon b(x)}\, l(\gamma) \leq l_\epsilon(\gamma) \leq \exp\{2\lambda\epsilon+8\epsilon\} e^{-\epsilon b(x)}\, l(\gamma).
\end{equation}

On the other hand, for any rectifiable curve $\beta\colon [0, L]\rightarrow Y $, where $L=l(\beta)$, with endpoints $x, y$, we have $l(\beta) \geq d_Y(x, y)$.
Choose $t_0\in [0, L]$ such that 
\[
l(\beta|_{[0, t_0]}) = d_Y(x, y).
\]
Therefore,
\[
\int_\beta  \rho_\epsilon^Y \mathrm{ds} \geq \int_0^{t_0} \rho_\epsilon^Y\! \left(\beta(t)\right) \mathrm{dt}.
\]
Observe that for any $z\in \beta|_{[0, t_0]}$, we have
\[
d_Y(x, z) \leq l(\beta|_{[0, t_0]}) \leq 4+\lambda.
\]
Again by Lemma \ref{allu_jose_004_009a}, we retain the relation \eqref{allu_jose_004_016} for all $z\in \beta|_{[0, t_0]}$.
\vspace{2mm}

Thus, a simple computation shows that
\[
\int_\beta  \rho_\epsilon^Y \mathrm{ds} \geq  d_Y(x, y) e^{-\epsilon b(x)} \exp\{ -2\lambda\epsilon -8\epsilon\},
\]
thereby taking the infimum over all curves $\beta$ joining $x$ and $y$ we obtain \[
d_\epsilon (x, y)  \geq d_Y(x, y) e^{-\epsilon b(x)} \exp\{ -2\lambda\epsilon -8\epsilon\}.
\]
Note that in \eqref{allu_jose_004_017}, $l(\gamma)= d_Y(x, y)$.\vspace{2mm}
 
Therefore, we have
\begin{align*}
d_Y(x, y) e^{-\epsilon b(x)} \exp\{2\lambda\epsilon+8\epsilon\} 
&\geq l_\epsilon(\gamma) \geq d_\epsilon(x, y) \\
&\geq d_Y(x, y) e^{-\epsilon b(x)} \exp\{-2\lambda\epsilon -8\epsilon\}.
\end{align*}
Thus, we conclude that
\begin{equation}\label{allu_jose_004_018}
d_\epsilon(x, y) \asymp e^{-\epsilon b(x)} d_Y(x, y).
\end{equation}
Moreover, from \eqref{allu_jose_004_017} and \eqref{allu_jose_004_018} we observe that
\[ l_\epsilon(\gamma) \lesssim d_\epsilon(x, y).
\]
That is, $\gamma$ is a quasiconvex curve in $Y_\epsilon $ with constants depending on $\epsilon$ and $\lambda$.\vspace{2mm}

By Lemma \ref{allu_jose_004_010}, we have
\[
\delta_\epsilon(z) \gtrsim \rho_\epsilon(z).
\]
But by \eqref{allu_jose_004_016} and since $d_Y(x, y)\leq 4+\lambda$, we have 
\[
\rho_\epsilon(z) \gtrsim \rho_\epsilon(x) \gtrsim l_\epsilon(\gamma).
\]
Hence, combining the above two relations we obtain 
\[ 
\delta_\epsilon(z)\gtrsim l_\epsilon(\gamma),
\]
which shows that $\gamma$ is a double cone curve in $Y_\epsilon$ with constant depending on $\epsilon,$ and  $\lambda$. Thus, $\gamma$ is a uniform curve. this completes the proof.
\end{proof}
\begin{lemma}\label{allu_jose_004_019}
Let $x, y\in Y$ be such that $d_Y(x, y) \geq 2+\lambda$. Then
\[
d_\epsilon(x, y) \asymp d_\epsilon\left(\phi(x), \phi(y)\right).
\]
\end{lemma}
\begin{proof}
Let $\gamma \colon [0,L] \to Y$ be a curve joining $x$ to $y$ with length $L\geq 2+\lambda\geq 2$. 
Fix $N\in \mathbb{N}$ satisfying $N\leq L\leq N+1$ and set $q=\frac{L}{N}$. 
For $0\leq i \leq N$, define 
\[
a_i= \gamma(iq) \quad \text{and} \quad b_i=\phi(a_i).
\]
Then we have
\[
d_X(b_i, b_{i+1}) \leq d_Y(a_i, a_{i+1}) +\lambda \leq 4+\lambda.
\]
Hence, by Lemma \ref{allu_jose_004_015}, we obtain
\[
d_\epsilon(b_i, b_{i+1}) \lesssim e^{-\epsilon b(b_i)} = \rho_\epsilon^X(b_i).
\]
It follows that
\[
d_\epsilon (\phi(x), \phi(y)) \leq \sum_{i=0}^{N-1} d_\epsilon(b_i, b_{i+1}) \lesssim \sum_{i=0}^{N-1} \rho_\epsilon^X(b_i).
\]
But we have
\[ 
\sum_{i=0}^{N-1} \rho_\epsilon^X(b_i) \asymp \int_\gamma \rho_\epsilon^Y \mathrm{ds},
\]
so taking the infimum over all such curves gives
\[
d_\epsilon (\phi(x), \phi(y)) \lesssim d_\epsilon(x, y).
\]
On the other hand, we have
\begin{equation}\label{allu_jose_004_020}
d_X(\phi(x), \phi(y)) \geq d_Y(x, y) -\lambda\geq 2.
\end{equation}
Set $x^\prime=\phi^{-1}\left(\phi(x)\right)$ and $y^\prime = \phi^{-1}\left(\phi(y)\right)$.  
An argument analogous  to the above yields 
\[
d_\epsilon(x^\prime, y^\prime) \lesssim d_\epsilon(\phi(x), \phi(y)).
\]
By definition of $\phi^{-1}$, for each $z\in Y$ we have  $d_Y(z, \phi^{-1}(\phi(z))) \leq \lambda$ . 
Then it immediately follows from Lemma \ref{allu_jose_004_015} that 
\[
d_\epsilon(x, x^\prime) \lesssim e^{-\epsilon b(x)} \quad \text{and} \quad d_\epsilon(y, y^\prime) \lesssim e^{-\epsilon b(y)}.
\]
Therefore, we have
\begin{align}
d_\epsilon(x, y) &\leq d_\epsilon(x, x^\prime)+ d_\epsilon(x^\prime, y^\prime) + d_\epsilon(y^\prime, y)\nonumber\\
&\lesssim e^{-\epsilon b(x)} + d_\epsilon(\phi(x), \phi(y)) + e^{-\epsilon b(y)} \label{allu_jose_004_021}.
\end{align}
Moreover, from Lemma \ref{allu_jose_004_006a}, for any curve $\beta$ joining $\phi(x)$ and $\phi(y)$, we have
\begin{align*}
\int_\beta\rho_\epsilon^X \mathrm{ds} &\asymp  \sum_{i=0}^{N-1} \rho_\epsilon^X(b_i) \geq  \rho_\epsilon^X(\phi(x)) + \rho_\epsilon^X(\phi(a_{N-1})) \\
&\gtrsim \rho_\epsilon^X (\phi(x)) + \rho_\epsilon^X (\phi(y)).
\end{align*}
Taking the infimum yields us
\begin{equation}\label{allu_jose_004_022}
d_\epsilon(\phi(x), \phi(y))\gtrsim e^{-\epsilon b^\prime(\phi(x))} + e^{-\epsilon b^\prime(\phi(y))}.
\end{equation}
From \eqref{allu_jose_004_009c}, for every $z \in Y$ we have
\begin{equation}\label{allu_jose_004_023}
e^{-\epsilon b(z)} \lesssim e^{-\epsilon b^\prime(\phi(z))}. 
\end{equation}
Therefore, combining \eqref{allu_jose_004_023} and 
\eqref{allu_jose_004_022} with \eqref{allu_jose_004_021}, we conclude
\begin{align*}
d_\epsilon(x, y) &\lesssim  e^{-\epsilon b^\prime(\phi(x))} +d_\epsilon(\phi(x), \phi(y)) + e^{-\epsilon b^\prime(\phi(y))}\\
& \lesssim d_\epsilon(\phi(x), \phi(y)).
\end{align*}
Hence, the proof is complete.
\end{proof}
\section{Main result and proof}\label{allu_jose_004_024a}

Let $Y$ be a $\delta$-hyperbolic space that is $K$-roughly starlike with respect to $\omega\in \partial_\infty Y$.
Further, let $\phi :Y \to X$ be a $\lambda$-rough isometry such that $\phi^{-1}$ is also $\lambda$-rough isometry.
Then an argument similar to that in \cite[Remark 2.9]{lindquist_shanmu_2021} shows that $X$ is $K^\prime$-roughly starlike with respect to $\omega^\prime=\phi(\omega)\in \partial_\infty X$, where $K' = K^\prime(\delta, \lambda, K)$. Therefore, in the subsequent theorem, it suffices to assume that $X$ is roughly starlike.

\begin{theorem}\label{allu_jose_004_024}
Let $\phi \colon Y \to X$ be a $\lambda$-rough isometry between two proper Gromov hyperbolic spaces $(X, d_X)$ and $(Y, d_Y)$. Suppose $\omega\in \partial_\infty Y$ and let $\omega^\prime = \phi(\omega) \in \partial_\infty X$.
If $X$ is roughly starlike with respect to $\omega^\prime$ and $\epsilon>0$ is such that $(X_\epsilon, d_\epsilon)$ is a uniform space, then $(Y_\epsilon, d_\epsilon)$ is also a uniform space.
\end{theorem}
\begin{proof}
Let $x, y\in Y$. 
If $d_Y(x, y) \leq 4+\lambda$, then by Lemma \ref{allu_jose_004_015} we know that the hyperbolic geodesic joining $x$ to $y$ is a uniform curve in $(Y_\epsilon, d_\epsilon)$.
Therefore, it suffices to consider points $x, y\in Y$ satisfying $d_Y(x, y)\geq 4+\lambda$.
For such points $x, y$ we have $d_X(\phi(x), \phi(y))\geq 4$. 
Since $(X_\epsilon, d_\epsilon) $ is a uniform space, let $\gamma$ be a uniform curve in $X_\epsilon$ joining $\phi(x)$ to $\phi(y)$.
Then it is clear that $l(\gamma)\geq 4$.
Therefore, by Lemma \ref{allu_jose_004_006a} and Lemma \ref{allu_jose_004_019} we have
\[
d_\epsilon(x, y) \asymp d_\epsilon(\phi(x), \phi(y)) \asymp \int_\gamma \rho_\epsilon^Y \mathrm{ds}.
\]
We set $b_i =\gamma(iq),\, q=\tfrac{L}{N},\, L\coloneq l(\gamma) \text{ and } N\leq L <N+1$.
Now, we apply Lemma \ref{allu_jose_004_009e} to $X$ with $\phi^{-1}$ instead of $\phi$. 
In conjunction with the previous estimate we obtain
\[ 
d_\epsilon(x, y) \asymp \sum_{i=0}^{N-2} \rho_\epsilon^Y \left(\phi^{-1}(b_i)\right) + \rho_\epsilon^Y\left(\phi^{-1}\left(\phi(y\right)\right).
\]
Note that, $d_Y(y, \phi^{-1}\left(\phi(y\right)) \leq \lambda$ and $d_Y(x, \phi^{-1}\left(\phi(x\right) \leq \lambda$. Therefore, we obtain,
\[
d_\epsilon(x, y) \asymp \rho_\epsilon^Y\left(\phi^{-1}\left(\phi(x\right)\right) +\sum_{i=1}^{N-2} \rho_\epsilon^Y \left(\phi^{-1}(b_i)\right) + \rho_\epsilon^Y\left(\phi^{-1}\left(\phi(y\right)\right).
\]
Moreover, we observe that
\[
d_X(b_i, b_{i+1}) \leq l\left(\gamma|_{[iq, (i+1)q]}\right) =q \leq 1.
\]
Thus, we have
\[
d_Y \left(\phi^{-1}(b_i), \phi^{-1}(b_{i+1})\right) \leq 1+\lambda.
\]
In a similar way, we see that
\[
d_Y \left(x, \phi^{-1}(b_1)\right) \leq 1+2\lambda \quad \text{and} \quad d_Y\left(y, \phi^{-1} (b_{N-1}) \right)\leq 1+2\lambda.
\]
Now, we construct a concatenation $\beta$ of geodesics in $Y$ and show that $\beta$ is our desired uniform curve.
Let $\beta_0$ be the hyperbolic geodesic in $Y$ joining $x $ to $\phi^{-1}(b_1)$ and $\beta_{N-1}$ be the hyperbolic geodesic joining $\phi^{-1}(b_{N-1})$ to $y$. 
For $i= 1,2, \cdots N-2$, set $\beta_i$ as the geodesic in $Y$ with endpoints $\phi^{-1}(b_i)$ and $\phi^{-1}(b_{i+1})$ and $\beta $ be the concatenation of $\beta_i $ for $i = 0, 1, \cdots N-1$. Therefore, by Lemma \ref{allu_jose_004_015}, for $i= 1,2, \cdots N-2 $ we obtain
\begin{align*}
l_\epsilon(\beta_i) &\asymp \rho_\epsilon\left(\phi^{-1}(b_i)\right) 
d_Y\left(\phi^{-1}(b_i), \phi^{-1}(b_{i+1})\right)\\
&\lesssim \rho_\epsilon^Y\left(\phi^{-1}(b_i)\right).
\end{align*}
Similarly, for $\beta_0$ and $\beta_{N-1}$ we have
\[
l_\epsilon(\beta_0) \lesssim \rho_\epsilon^Y(x) \quad \text{and} \quad l_\epsilon(\beta_{N-1}) \lesssim \rho_\epsilon^Y(y).
\]
Therefore, we deduce that
\begin{align*}
d_\epsilon(x, y) &\asymp \rho_\epsilon^Y(x) + \sum_{i=1}^{N-2} \rho_\epsilon^Y \left(\phi^{-1}(b_i)\right) + \rho_\epsilon^Y(y)\\
& \gtrsim  \sum_{i=0}^{N-1} l_\epsilon(\beta_i) = l_\epsilon(\beta).
\end{align*}
Thus $\beta$ is a quasiconvex curve in $Y_\epsilon$ with endpoints $x$ and $y$.\vspace{2mm}

It remains to show that $\gamma$ satisfies the double cone condition. Let $z\in \beta$. If $z\in \beta_0$ then
\[
\min \left\{ l_\epsilon\left(\beta[x, z]\right), l_\epsilon\left(\beta[z, y]\right)\right\} = l_\epsilon \left(\beta[x, z]\right).
\]
Also, $d_Y(x, z) \leq 1+2\lambda$. Therefore, by similar arguments as in Lemma \ref{allu_jose_004_015} and by applying Lemma \ref{allu_jose_004_010}, we obtain that 
\[
\delta_\epsilon(z) \gtrsim \rho_\epsilon^Y (z) \gtrsim l_\epsilon\left(\beta[x, z]\right).
\]
Similarly, the conclusion holds if $z\in \beta_{N-1}$. Thus, we may assume that $z\in \beta_i$ for some $1\leq i \leq N-2$.
Since $\gamma $ is a uniform curve joining $\phi(x)$ to $\phi(y)$, we have 
\[
\min \left\{ l_\epsilon \left(\gamma[\phi(x), b_i]\right), l_\epsilon \left(\gamma[b_i, \phi(y)]\right)\right\} \leq A \delta_\epsilon(b_i).
\]
We first assume that 
\[ 
\min \left\{ l_\epsilon \left(\gamma[\phi(x), b_i]\right), l_\epsilon \left(\gamma[b_i, \phi(y)]\right)\right\} = l_\epsilon \left(\gamma[\phi(x), b_i]\right).
\]
Then, by Lemma \ref{allu_jose_004_010} we obtain
\[
\delta_\epsilon(z) \asymp \delta_\epsilon\left(\phi^{-1}(b_i)\right) \asymp \delta_\epsilon\left(b_i\right)\geq \frac{1}{A} l_\epsilon \left(\gamma[\phi(x), b_i]\right).
\]
Moreover, we have
\begin{equation}\label{allu_jose_004_025}
l_\epsilon\left(\gamma[\phi(x), a_i]\right) \asymp \sum_{j=0}^i \rho_\epsilon^X(a_j) \asymp \sum_{j=0}^i \rho_\epsilon^Y \left(\phi^{-1}(b_j)\right) \gtrsim l_\epsilon\left(\beta[x, z]\right).
\end{equation}
Thus we conclude that 
\[
\delta_\epsilon(z) \gtrsim l_\epsilon\left(\beta[x, z]\right).
\]
If
 \[
\min \left\{ l_\epsilon \left(\gamma[\phi(x), b_i]\right), l_\epsilon \left(\gamma[b_i, \phi(y)]\right)\right\} = l_\epsilon \left(\gamma[b_i, \phi(y)]\right)
\]
then we change the roles of $x$ and $y$ and take the sum over all $j$ from $i$ to $N$ in the estimate \eqref{allu_jose_004_025} to reach the conclusion.
Hence, consolidating all the above estimates we obtain
\[
\min\left\{ l_\epsilon\left(\beta[x, z]\right), l_\epsilon\left(\beta[z, y]\right)\right\} \lesssim \delta_\epsilon(z).
\]
This completes the proof.
\end{proof}

\section{Failure of the result for $(\kappa, \tau)$-rough similarities} \label{allu_jose_004_026}
A map $\phi : Y \to X$ is said to be a $(\kappa, \tau)$-rough similarity if $\kappa >0$ and $\tau \geq 0$ are such that 
\[
\kappa d_Y(x, y) - \tau \leq d_X\left(\phi(x), \phi(y)\right) \leq \kappa d_Y(x, y) + \tau,
\]
for every $x, y \in Y$, and the Hausdorff distance between $\phi(Y)$ and $X$ is at most $\tau$. 
Since rough isometries are $(1, \tau)$- rough similarities, it is natural to ask whether the conclusion of our main result, Theorem \ref{allu_jose_004_024}, remains valid if we replace rough isometries with rough similarities. 
That is, if both $X$ and $Y$ are Gromov hyperbolic spaces and $\epsilon >0$ is such that $X_\epsilon$ is uniform, and $\phi : Y \to X$ is a rough similarity, is $Y_\epsilon$ also a uniform domain? 
In this section we provide an example to show that this is not the case.
\vspace{2mm}

Let $X$ be the upper half plane $\mathbb{H}$ equipped with the hyperbolic metric $h$ induced by the density $\theta(z)={\mathrm{Im}(z)}^{-1}$. 
It is known that $X$ is Gromov $\delta$-hyperbolic with constant $\delta= \log 3$. Choose $b(x+\iota y) = -\mathrm{ln }\,\, y$, as the Busemann function based at $\infty\in \partial_\infty X$, and set $\rho_\epsilon^X (z) = e^{-\epsilon b(z)}$.
Then by \cite[Theorem 1.2]{zhou_ponnusamy_rasila_2025}, we see that there exists $\epsilon_0>0$ such that $X_{\epsilon_0}$ is an unbounded uniform space. 
For a fixed $\epsilon>0$, choose $Y$ as the same upper half plane equipped with the scaled hyperbolic metric $h^\prime(z_1, z_2) = \frac{\epsilon_0}{\epsilon}h(z_1, z_2)$. 
It follows that the identity map  $\mathrm{id}:\left(Y, h^\prime\right) \to \left(X, h\right)$ is a $\left(\frac{\epsilon_0}{\epsilon}, 0\right)$-rough similarity. 
Furthermore, it is easy to see that $Y_\epsilon = X_{\epsilon_0}$, which implies that $Y_\epsilon$ is a uniform space.
However, for  $\epsilon =2$, we show that $X_\epsilon$ is not a uniform space. \vspace{2mm}

We begin by observing that a hyperbolic circle centred at $\iota$ with radius $r$ has hyperbolic length 
\[
l_h(C_r) = \pi \left(e^r - e^{-r}\right).
\]
For $R\gg 1$, let
\[
z_R = \left(\frac{e^{2R}-1}{e^{2R}+1}, \frac{2e^{2R}}{e^{2R}+1}\right) \quad \text{and} \quad w_R = \left(\frac{e^{-2R}-1}{e^{-2R}+1}, \frac{2e^{-2R}}{e^{-2R}+1}\right).
\]
A simple calculation shows that the length of the geodesic ray joining $z_R$ to $1$ has length $\frac{2}{e^{2R}+1}$ and $\partial X_\epsilon$ has only one point, where $\epsilon = 2$. Therefore, 
\[
d_2 (z_R, w_R) \leq \frac{4}{e^{2R}+1}.
\]
If possible suppose that $X_\epsilon$ is $A$-uniform space for some constant $A\geq 1$. Let $\gamma$ be a uniform curve joining $z_R$ to $w_R$. 
Since $\gamma$ is compact, it is contained in a closed hyperbolic disk centred at $\iota$. 
Let $T$ be the smallest radius of such a disk. Thus, there must be a point $z^\prime \in \gamma$ such that 
\[
h(z^\prime, \iota) = T.
\]
Without loss of generality, we may assume that  
\[
l_\epsilon\!\left(\gamma[z_R, z^\prime]\right) \leq l_\epsilon\! \left(\gamma[z^\prime, w_R]\right).
\]
Then, by \eqref{allu_jose_004_006} we obtain
\begin{align*}
l_\epsilon\!\left(\gamma[z_R, z^\prime]\right) &= \int_{\gamma[z_R, z^\prime]} \rho_\epsilon^X |\mathrm{du}|\\
&\geq \exp\left\{-20\delta\right\}\rho_\epsilon^X(\iota)\int_R^T e^{-\epsilon t } \mathrm{dt}\\
&= 3^{-20} \frac{e^{-\epsilon R}-e^{-\epsilon T}}{\epsilon}.
\end{align*}
Since $\epsilon=2$, the double cone condition  \eqref{allu_jose_004_001b} implies
\[
3^{-20} \frac{e^{-2 R}-e^{-2 T}}{2} \leq \frac{2}{e^{2T}}.
\]

Therefore, we must have 
\[
T\leq R + \frac{\mathrm{ln}\left(4\cdot 3^{20}A +1\right)}{2}.
\]
Let $T^\prime$ be the largest number such that $\gamma$ lies in the complement of the hyperbolic ball centred at $\iota$ with radius $T^\prime$.
Then, similarly as before, we obtain
\[
l_\epsilon(\gamma) \geq 3^{-20} \frac{e^{-\epsilon T^\prime}-e^{-\epsilon R}}{\epsilon}.
\]
However, by the quasiconvex condition \eqref{allu_jose_004_001a} of a uniform arc for $\epsilon =2$ we obtain 
\[
3^{-20} \frac{e^{-2 T^\prime}-e^{-2 R}}{2}\leq A d_2 (z_R, w_R) \leq \frac{4A}{e^{2R}+1} \leq \frac{4A}{e^{2R}}.
\]
Simplifying the above expression yields
\[
T^\prime \geq R- \frac{\mathrm{ln}\left(8\cdot 3^{20}A+1\right)}{2}.
\]
Thus, we obtain that the curve is lying inside the annulus centred at $\iota$ with hyperbolic inner radius $R- {\mathrm{ln}\left(8\cdot 3^{20}A+1\right)}/{2}$ and hyperbolic outer radius $R + {\mathrm{ln}\left(4\cdot 3^{20}A +1\right)}/{2}$.
\vspace{2mm}

Thus, using the fact that$R\gg 1$ we obtain
\begin{align*}
l_2(\gamma) &= \int_\gamma \rho_2^X (u) |\mathrm{du}| \\
&\geq 3^{-20} \exp\left\{-2 \left(R + \frac{\mathrm{ln}\left(4\cdot 3^{20}A +1\right)}{2}\right)\right\}l_h(\gamma)\\
&\geq 3^{-20} \frac{\pi}{4} e^{-R} \left(4\cdot 3^{20}A +1\right)^{-1} \left(8\cdot 3^{20}A+1\right)^{-1/2}.
\end{align*}
Combining with the quasiconvexity \eqref{allu_jose_004_001a} of $\gamma$ we see that 
\[
C_A e^{-R} \leq A d_2 (z_R, w_R) \leq \frac{4A}{e^{2R}+1} \leq {4A}{e^{-2R}},
\]
where $C_A = 3^{-20} \frac{\pi}{4} \left(4\cdot 3^{20}A +1\right)^{-1} \left(8\cdot 3^{20}A+1\right)^{-1/2} $.
But this is not possible for large $R$. Since the uniformity constant $A$ was arbitrary, we conclude that the space $X_\epsilon$ for $\epsilon =2$ is not uniform. 
\vspace{2mm}

\noindent{\bf Acknowledgement.}
The first author thanks SERB-CRG and the second author's research work is supported by CSIR-UGC.\\
\noindent\textbf{Compliance of Ethical Standards:}\\
\noindent\textbf{Conflict of interest.} The authors declare that there is no conflict  of interest regarding the publication of this paper.\vspace{1.5mm}\\
\noindent\textbf{Data availability statement.}  Data sharing is not applicable to this article as no datasets were generated or analyzed during the current study.\vspace{1.5mm}\\
\noindent\textbf{Authors contributions.} Both the authors have made equal contributions in reading, writing, and preparing the manuscript.

\end{document}